\documentclass[12pt]{amsart}
\usepackage{a4}
\usepackage{amssymb}
\usepackage{amsmath}
\usepackage{amsthm}
\usepackage{amstext}
\usepackage{amscd}
\usepackage{latexsym}
\usepackage{graphics}
\usepackage{color}
\swapnumbers
\theoremstyle{plain}
\newtheorem{thm}{Theorem}[section]

\newtheorem{lemma}[thm]{Lemma}

\newtheorem{corollary}[thm]{Corollary}
\newtheorem{prop}[thm]{Proposition}
\newtheorem{conj}[thm]{Conjecture}
\theoremstyle{definition}

\newtheorem{point}[thm]{}
\newtheorem{rmk}[thm]{Remark}

\newtheorem{defn}[thm]{Definition}

\newcommand{\Spec}{{\rm Spec \,}}

\renewcommand{\tilde}{\widetilde}
\newcommand{\A}{{\mathbb A}}

\renewcommand{\P}{{\mathbb P}}

\newcommand{\R}{{\mathbb R}}

\newcommand{\Z}{{\mathbb Z}}

\input{xy}
\xyoption{all}

\begin{document}
\title{Weak approximation for low degree Del Pezzo surfaces }
\author{Chenyang Xu}

\begin{abstract}
Let $K={\rm Func}(C)$ be the function field of a smooth curve $C$. For every Del Pezzo surface $S/K$
which is appropriately generic, weak approximation for $S$ holds at every place of $K$, i.e., for every closed point $c$ of $C$. This combines earlier work in \cite{xu} with an analysis of weak approximation near boundary points of
the parameter spaces for Del Pezzo surfaces of degrees 1 and 2.
\end{abstract}
\maketitle

%%%%%%%%%%%%%%%%%%%%%%%%%%%%%

\tableofcontents

%%%%%%%%%%%%%%%%%%%%%%%%%%%%%%%%%%%%%%%%%%%%%%%%%%%%%%%%%%%%%%%%%%%%%%%%
\vspace*{6pt}
\section{Introduction}
Throughout this paper, the ground field is always of
characteristic 0. For a variety $X$ defined over a number field
$K$, it is a classical question to study the existence and
distribution  of $K$-points on $X$. We say that $X$ satisfies {\it weak
approximation} if for any finite set of places of $K$ and points of
$X$ over the completion of $K$ at these places, there exist
$K$-rational points of $X$ in any neighborhood (in the adelic
topology) of these points. In this note, instead of a number field
we study varieties defined over the function field $K=k(C)$ with $k$
algebraically closed and $C$ a smooth curve. In this context,
rational points correspond to sections of fibrations over a curve,
and proving weak approximation corresponds to finding sections with
prescribed jet data in a finite number of fibers.

 The existence of sections of rationally connected
fibrations was proven by Graber, Harris and Starr in \cite{ghs}.
Under the assumption that there exists a section, Koll\'ar, Miyaoka
and Mori proved the existence of sections through a finite set of
prescribed points in smooth fibers (cf. \cite{kmm}, 2.13 and
\cite{kollarrc}, IV.6.10). The existence of sections with prescribed
finite jet data through smooth fibers, i.e., weak approximation at
places of good reduction, was proven by Hassett and Tschinkel in
\cite{ht06}. In the same paper Hassett and Tschinkel made the
following conjecture.
\begin{conj}[\cite{ht06}]\label{wkconj}
A smooth rationally connected variety $X$ defined over the function field $K$
satisfies weak approximation at places of bad reduction.
\end{conj}

When $X$ is a surface, Colliot-Th\'el\`ene and Gille proved that conic
bundles over $\P^1$ and Del Pezzo surfaces of degree at least four satisfy weak approximation at all places. The cases of Del Pezzo surfaces of degree less than four are still open. It is known that cubic surfaces with square-free discriminant satisfy weak approximation even at places of bad reduction (cf. \cite{ht08}). And a similar result is generalized to degree 2 Del Pezzo surfaces by Knecht \cite{kn}. This paper addresses weak approximation for more cases of low degree Del Pezzo surfaces, including cases of degree 1 Del Pezzo surfaces.

\begin{point}{\bf Notations:}
We first explain some terminology. Let $S/K$ be a smooth Del Pezzo surface of degree 1. Then $S$ can be embedded in the weighted projective space $\P_K(1,1,2,3)$ as a degree 6 hypersurface (for the meaning of the notation, see e.g. \cite{kollarrc}, V.1.3). In the case that $K$ is the fraction field of a smooth curve $C/k$, we can take the closure $\mathcal{S}$ of $S$  in $\P_C(1,1,2,3)$, so that $\mathcal{S}$ is a flat family of degree 6 hypersurfaces in $\P(1,1,2,3)$ over $C$. We denote by $\P^N$ the space which parametrizes all degree 6 hypersurfaces in $\P(1,1,2,3)$. Because $\P(1,1,2,3)$ has two singular points $P_I=(0,0,1,0)$ and $P_{II}=(0,0,0,1)$ which are quotient singularities and of types $\frac{1}{2}(1,1,1)$ and $\frac{1}{3}(1,1,2)$, the locus parametrizing the singular surfaces consists of 3 irreducible components. Two of them are the hyperplanes $H_I$ and $H_{II}$ which parametrize hypersurfaces containing $P_I$ and $P_{II}$.  $H_I$ (resp. $H_{II}$) has a dense open set $H_I^0$ (resp. $H_{II}^0$) parametrizing surfaces with only one singularity which is quotient and of type $\frac{1}{4}(1,1)$ (resp. $\frac{1}{9}(1,2)$) (see Section 3 for more details of the computation). The last component is a hypersurface $A$ which is the closure of the discriminant divisor $A^0$. Here $A^0$ parametrizes singular degree 6 hypersurfaces which contain neither $(0,0,1,0)$ nor $(0,0,0,1)$. Then $A^0$ is the image of a birational  morphism from a smooth variety of dimension $N-1$ (see e.g., \cite{vo03}, 2.1.1) and the smooth locus $(A^{0})^{sm}$ parametrizes surfaces with only one ordinary double point.
\end{point}

\begin{defn}\label{general}
Let $C$ be a smooth curve over $k$. We say that $S/K(=k(C))$ admits a model $\mathcal{S}/C$ which is {\it a transversal family of degree 1 Del Pezzo surfaces} over $C$, if the morphism $\Spec(K)\to \P^N$, which gives the point $[S]$ in the parametrizing space, can be completed to a morphism $f:C\to \P^N$ such that:
\begin{enumerate}
 \item $C$ meets $H_I$ (resp. $H_{II}$) transversally in $H_I^0$ (resp. $H_{II}^0$); and
\item $C$ intersects $A$ in $A^0$ and meets each branch of $A^0$ transversally.
\end{enumerate}
We note that because $H_I$ and $H_{II}$ are open sets of $\P^{N-1}$,  they are smooth, and our transversality condition in $(1)$ is the usual one. The condition $(2)$ means that each branch of $A^0$ at a point $c\in C\cap A^0$ is smooth and transverse to $C$. Also see \eqref{duval}. 
\end{defn}

From the discussion above it is obvious that a generic one-parameter family of degree 1 Del Pezzo surfaces in $\P^N$ is a transversal family.

\begin{thm}\label{main} For every smooth
     degree 1 Del Pezzo surface $S$ defined over the function field $K$ of the curve
    $ $C, if there exists a model giving a transversal family of degree 1 Del Pezzo surfaces $\mathcal{S}$ over $C$, then weak approximation holds at each place $c\in C$.
\end{thm}

\begin{rmk}
We consider all degree 2 Del Pezzo surfaces to be embedded in the weighted projective space $\P(1,1,1,2)$. The locus parametrizing singular fibers consists of two components: $H$ for surfaces containing $(0,0,0,1)$ and $A$ the discriminant.  Then we can similarly define $H^0$, $A^0$ and {\it a transversal family} of degree 2 Del Pezzo surfaces over $C$. Applying the approach in this note to this simpler case, we can prove a similar statement as above. The details are left to the reader.
\end{rmk}

Now let us explain our approach. In \cite{ht06} and \cite{ht08}, Hassett and Tschinkel initiated the method of establishing weak approximation by showing the strong rational connectedness of the smooth locus of the special fibers. In \cite{xu}, we showed that the smooth loci of log Del Pezzo surfaces are always strongly rationally connected. However, for a given low degree Del Pezzo surface $S/K$, usually we can not expect the existence of a smooth model such that the fibers are all log Del Pezzo surfaces. We have to resolve the singularities, thus there is more than one component in some special fibers. To deform an initial section to one matching the prescribed jet data,  we have to find some auxiliary curves to correct the intersection numbers. A similar technique was applied in \cite{ht} to study weak approximation for places where the fibers only contain ordinary singularities.

\begin{rmk}
In \cite{co96}, Corti established a theory of {\it good models} for Del Pezzo surfaces. We note that if $\mathcal{S}/C$ is a transversal family in our sense, then it gives a good model over each point $c\in C$. In fact, from the local rigidity (\cite{co96}, Theorem 1.18), we know the model we are dealing with in this note is the `best' model. It is natural to ask whether a similar approach works for other good models as well.
 \end{rmk}

\noindent
{\bf Acknowledgement:} I would like to thank Brendan Hassett, J\'anos Koll\'ar, Jason Starr for their helpful conversations and emails. I am also indebted to the anonymous referee for many comments which have substantially improved the exposition. In particular, Subsection 3.1 is reorganized according to the referee's suggestions.  Part of the work was done during the author's stay in the Institute for Advanced Study, which was supported by the NSF under agreement No. DMS-0635607. The author was partially supported by NSF research grant no: 0969495.

\noindent {\bf Notations and Conventions:} In the following, a letter in script such as $\mathcal{X}$ always means the model. And for a closed point $c$, we will use the Roman letter $X_c$ to mean the fiber of the model $\mathcal{X}$ over $c$.

\section{Weak Approximation and Rational Curves }

 In this section, we will briefly recall the background. For more discussions, see \cite{ht06} and \cite{ht08}.

\subsection{Weak approximation}
\begin{defn} Let $K$ be  a number field or the function field of a curve $C$ defined over an algebraically closed field $k$. Let $B$ a finite set of places of $K$ containing the archimedean places.
Let $X$ be a variety over $K$, $X(K)$ the set of $K$-rational
points. One says that {\it weak
approximation holds for $X$ away from $B$} if $X(K)\subset \prod_{v\not\in B} X(K_v)$ is dense in the natural direct product topology, where the topology of each fiber $X(K_v)$ is induced by the discrete topology on $X_c(k)$.
\end{defn}

From the method we use, it is easy to see that for our setting, to show $X/K$ satisfies weak approximation away from $B$, it suffices to show for each place $c\not\in B$, weak approximation holds there, i.e., we can work one place at a time.

\subsection{Strong rational connectedness}

The following concept is a variant of rational connectedness in the case that the variety is smooth but nonproper.

\begin{defn}[\cite{ht08}, 14]
If $X$ is a smooth variety over an algebraically closed field $k$, then $X$ is called {\it strongly rationally connected}  if for each point $x\in X$ there is a morphism $f:\P^1 \to X$ such that:
\begin{enumerate}
 \item $x\in f(\P^1)$;
\item $f^*(T_X)$ is ample.
\end{enumerate}
\end{defn}

\noindent
The relationship between weak approximation and strong rational connectedness is established by the following theorem due to Hassett and Tschinkel.

\begin{thm}[\cite{ht06}, \cite{ht08}]\label{wa}
Let $X$ be a smooth proper rationally connected variety over
$K=k(C)$, where $C$ is a smooth curve. Let $\pi : \mathcal{X} \to C$ be a proper smooth model of $X$.  Let $\mathcal{X}^{sm}$ be the locus where $\pi$ is smooth  such that
\begin{enumerate}
\item there exists a section $s : C \to \mathcal{X}^{sm}$;
\item  for each $c \in C$ and $x \in X^{sm}_c$ the fiber over $c$, there exists a rational curve $f : \P^1 \to
X^{sm}_c$ containing $x$ and a general point of $X^{sm}_c$.
\end{enumerate}
Then weak approximation holds for each place $c\in C$.
\end{thm}

For strong rational connectedness of surfaces, we know the following result.

\begin{thm}[\cite{xu}]\label{src}
Let $S$ be a log Del Pezzo surface, i.e., $S$ only has quotient singularities and $K_S$ is anti-ample. Then its smooth locus $S^{sm}$ is strongly rationally connected.
\end{thm}

On the other hand, we also have the following.

\begin{lemma}\label{duval}
 Notations as in (\ref{main}). A point $c\in C\cap A$ satisfies the assumption (2) of (\ref{general}) if and only if $\mathcal{S}$ is smooth and the fiber $S_c$ of $\mathcal{S}$ over $c$ contains at worst  $A_1$ singularities.
\end{lemma}
\begin{proof}
 See the proof of (\cite{ht08}, 23).
\end{proof}

Putting these two results together, for points in $A$, we have

\begin{corollary}
 Let $S/k(C)$ be a smooth log Del Pezzo surface of degree 1, where $C$ is a smooth curve. Assume $S$ admits a model giving a transversal family over $C$. If $c$ is a point in $C\cap A^0$, then $S$ satisfies weak approximation at the place $c$.
\end{corollary}

\subsection{Deformation theory}\label{deformation} In our discussion, we need to apply the technique of smoothing a nodal curve to a smooth curve. This was first used in \cite{kmm}, and then improved in \cite{ghs}. For the proof of the statements in this section, see \cite{ht06}, \cite{ht08} and \cite{ht}.

\begin{defn}
 A projective nodal curve $C$ is {\it tree-like} if
\begin{enumerate}
 \item each irreducible component of $C$ is smooth; and
\item the dual graph of $C$ is a tree.
\end{enumerate}
\end{defn}

 Let $f:C\to  Y$ denote an unramified morphism
whose image is a nodal curve. The restriction homomorphism
$f^*\Omega^1_Y \to \Omega^1_C$ is surjective and the dual to its kernel is still locally free. This
is denoted by $\mathcal{N}_f$ and coincides with $\mathcal{N}_{C/Y}$ when $f$ is an embedding. First order
deformations of $f: C\to Y$ are given by $H^0(C,\mathcal{N}_f)$; obstructions are given by $H^1(C,\mathcal{N}_f)$.
When $D$ is a union of irreducible components of $C$, the analogous extension takes
the form:
$$0 \to \mathcal{N}_f|D \to \mathcal{N}_f \otimes \mathcal{O}_D \to Q \to 0.$$
Here $Q$ is a torsion sheaf supported on the locus where $D^c:=\overline{C\setminus D}$ meets $D$, with length one at each point in the locus.

\begin{prop}[\cite{ht06},  24]\label{smoothing}
  Let $C$ be a tree-like curve, $Y$ a smooth algebraic space,
and $f : C \to Y$ an immersion with nodal image. Suppose that for each
irreducible component $C_l$ of $C$, $H^1(C_l,\mathcal{N}_f\otimes \mathcal{O}_{C_l}) = 0$ and $\mathcal{N}_f \otimes \mathcal{O}_{C_l}$ is globally generated. Then $f : C \to Y$ deforms to an immersion of a smooth curve into $Y$.
Suppose furthermore that $P = \{p_1, . . . , p_w\} \subset C$ is a collection of
smooth points such that for each component $C_l$, $H^1(\mathcal{N}_f \otimes \mathcal{O}_{C_l}(-P)) = 0$
and the sheaf $\mathcal{N}_f \otimes \mathcal{O}_{C_l}(-P)$ is globally generated. Then $f : C \to Y$
deforms to an immersion of a smooth curve into $Y$ containing $f(P)$.
\end{prop}

\section{The Proof of the Main Theorem}

In this section, we study the places of $C$ mapped into $H_I$ or $H_{II}$. 

\subsection{Singular fibers} We first introduce some standard notations for quotient singularities: A singularity is written as $\frac{1}{r}(a_1,a_2,...,a_n)$ if \'etale locally it is given by the quotient of the action of $\Z/r$ on $\A^n$ by
$$(x_1,x_2,...,x_n)\to (\xi^{a_1}x_1,\xi^{a_2}x_2,...,\xi^{a_n}x_n),$$
where $\xi$ is a primitive $r$-root of unity. We also assume $\gcd\{a_1,...,a_n,r\}=1$.

In the following, we will prove structural results for the singular fibers, which are parametrized by points in $H_I$ and $H_{II}$. We will use the following simple lemma repeatedly.
\begin{lemma}\label{simple}
Let $g:S_1 \to S_2$ be a birational morphism between smooth projective surfaces. Then for any irreducible curve $C$ on $S_1$, $C^2\le g(C)^2$.
\end{lemma}
\begin{proof}Left to the reader.
\end{proof}

\noindent{\bf Case }(I): For a place $c\in C\cap H_I$, because of the transversality assumption (\ref{general}.1), we know that the global family $\mathcal{S}$ over $\mathcal{O}_{c,C}$ has a unique quotient singularity of type $\frac{1}{2}(1,1,1)$. Furthermore, if we write the coordinates of $\P(1,1,2,3)$ as $(x,y,z,w)$, a point in $H_I$ gives a surface  with an equation without the term $z^3$. So a general equation in this form has a unique singularity at $P_I=(0,0,1,0)$. If we dehomogenize the equation with respect to $z$, the lowest degree term will have the form
$$wf_1(x,y)+f_2(x,y)=0,$$
where $f_1$ (resp. $f_2$) is a general linear (resp. quadratic) form  of variables $x$ and $y$. So the singularity of the fiber $S_c$ is of the form
$$(z^2=xy\subset \A^3)/(x,y,z)\to (-x,-y,-z),$$
which is of type $\frac{1}{4}(1,1)$. A single blow-up of $\mathcal{S}$ at the singular point $P_I$ gives a resolution $\pi:\mathcal{T}\to \mathcal{S}$, with an exceptional divisor $E_1\cong \P^2$, whose normal bundle is isomorphic to $\mathcal{O}(-2)$. Let $E_0$ be the birational transform of $S_c$, then $\pi|_{E_0}: E_0\to S_c$ is the minimal resolution of $S_c$, with the exceptional curve $R=E_0\cap E_1$ with self-intersection $-4$. Thus $R\subset E_1$ is a smooth conic. The configuration is illustrated in Figure 1.

\begin{center}
\setlength{\unitlength}{1mm}
\begin{picture}(100,60)
\multiput(30,30)(30,0){1}{\line(1,0){25}}
\multiput(30,30)(30,0){1}{\line(-1,-1){15}}
\multiput(30,30)(30,0){1}{\line(-1,1){15}}
\multiput(55,30)(30,0){1}{\line(1,1){15}}
\multiput(55,30)(30,0){1}{\line(1,-1){15}}
{\Small
\put(40,31){$R$}
\put(45,31){4}
\put(45,27){-4}

\put(42,15){$E_0$}
\put(38,40){$E_1\cong \P^2$}
}

\put(38, 5){\mbox{Figure 1}}
\end{picture}
\end{center}

In this case, $E_0$ is a surface such that the only $K_{E_0}$-negative curve is a smooth rational curve with self-intersection $-4$. 
\begin{lemma}\label{case1} There exists a birational contraction $\tilde{g}: E_0\to F_4$. In particular, $\tilde{g}(R)=B$, which is the section of $F_4 \to \P^1$ with the self-intersection number to be $-4$. \end{lemma}
\begin{proof}
We first remark that $R$ is the only curve on $E_0$ with the self-intersection less than $-1$. The surface minimal model program gives a morphism $g: E_0\to F_n$ for some Hirzebruch surface $F_n$. Moreover, we claim that in each step we can choose a $-1$-curve which is not  the image of $R$, i.e., $g(R)$ is a curve on $F_n$. In fact, the Mori cone of an intermediate surface $\tilde{F}$  has more than one extremal ray. If the image of $R$ on $\tilde{F}$ is a curve, then any extremal ray  $\R_{\ge 0}[C]$ which is not generated by the image of $R$ satisfies $K_{\tilde{F}}\cdot C<0$, thus after contracting $C$ we get a surface on which $R$ is not a point and we can proceed by induction.

 If the contraction $g$ gives a morphism $E_0\to F_n$ such that $g(R)$ is a section of $F_n\to \P^1$, then $E_0\to F_n$ factors through  $\tilde{F}$ where $\tilde{F}\to F_n$ is a morphism which blows up $4+g(R)^2$ distinct points on $g(R)\subset F_n$. Since the image of $R$ on $\tilde{F}$ is a section of $\tilde{F}$ with self-intersection $-4$ and $\tilde{F}$ does not have any other irreducible curve with self-intersection less than $-1$ except possibly $g(R)$ (cf. \eqref{simple}), we know that every fiber of $\tilde{F}\to \P^1$ has at most two components. By contracting the components in the reducible fibers which do not meet the image of $R$, it is easy to see that $\tilde{F}$ yields a morphism to $F_4$ such that the image of $R$ is the $-4$-curve on $F_4$. It also follows from \eqref{simple} that if $E_0$ yields a morphism to $F_n$ for some $n\le 2$, then the image of $R$ on $F_n$ must be the section with negative self-intersection. 
 
Thus we can assume $n=0$ or $1$ and in these cases we can indeed assume that there is a morphism $h$ giving a contraction from $E_0$ to $\P^2$.  Since we can assume $g(R)$ is not a section on $F_n$, we conclude that $R$ is not contracted by $h$.
It follows from $K^2_{E_0}=0$ and the fact that $R$ is the only curve whose intersection number with $K_{E_0}$ is nonnegative, we obtain that $g$ is the blowing up of 9 distinct points $P_i (1\le i \le 9)$ on $\P^2$. Let $D_i$ be the class of the exceptional curve which is contracted to $P_i$ by $g$. Then the class of $R$ is
$$g^*\mathcal{O}(m)-a_1D_1-\cdots-a_9D_9\mbox{ with }a_1\ge\cdots\ge a_9\ge 0. $$
and $K_{E_0}\sim g^*\mathcal{O}(3)+D_1+\cdots +D_9$. Computing $R^2$ and $R\cdot K_{E_0}$, we have
$$m^2-a_1^2-\dots-a_9^2=-4\mbox{ and }-3m+a_1+\cdots+a_9=2.$$ Eliminating the variable $m$,  we have the equation
$$9(a_1^2+\dots +a_9^2)=(a_1+\dots + a_9)^2-4(a_1+\dots +a_9)+40.$$
Since $$9(a_1^2+\dots +a_9^2)\ge(a_1+\dots + a_9)^2,$$ and
we have $5\le 3m+2=a_1+\dots+a_9<10$. The only solutions in nonnegative integers are
$$(a_1,\dots,a_9)=(1,\dots,1,0)\mbox{ or }(1,1,1,1,1,0,\dots,0),$$
and it is easy to get the conclusion in these cases. For example, in the case $(a_1,\dots,a_9)=(1,\dots,1,0)$, we can choose $\tilde{g}$ to be the morphism which contracts the birational transform of the lines connecting $P_iP_9$ ($1\le i \le 8$).

\end{proof}

\begin{corollary}\label{au1}
There exists a rational curve in $E_0$ meeting $R$ transversally in a single point and intersecting no other exceptional divisor of the morphism.

\end{corollary}

\begin{proof}

Since there exists $\tilde{g}:E_0\to F_4$, such that $\tilde{g}(R)=B$, we can choose the rational curve to be the birational transform of a general fiber of $F_4\to \P^1$.
\end{proof}

\begin{rmk}Let $\overline{F}_4$ be the singular surface which is obtained by contracting the $-4$-curve in $F_4$. We know that the morphism $\tilde{g}:E_0\to F_4$ yields a morphism $\overline{g}:S_c\to \overline{F}_4$, which is a blow-up of $8$ general (smooth) points on $\overline{F}_4$. Since $-1$-curves are rigid, after a possible \'etale base change of $C$ near $c$, we get 8 families of $-1$-curves in $\mathcal{S}$ over $C$. Contracting them, we obtain a morphism $\mathcal{S}\to \mathcal{S}^m$ such that $S^m_t\cong \P_t^2$ for  general geometric points $t$ in $C\setminus \{c\}$, and $S^m_c\cong \overline{F}_4$. In particular, we conclude that  $\overline{F}_4$ is a {\it Manetti surface}. For more discussion on related topics, see \cite{HP10}.
\end{rmk}

\noindent{\bf Case}(II): The singular point $(0,0,0,1)$ in $P(1,1,2,3)$ is of type $\frac{1}{3}(1,1,2)$. The equation of  a degree 6 surface passing through $(0,0,0,1)$ will have the form
$$wf_3(x,y,z)+f_6(x,y,z)=0,$$
where $f_3$ and $f_6$ are degree 3 and degree 6 homogeneous equations in $\P(1,1,2)$.
 If $f_3$ and $f_6$ are general, after dehomogenizing the equation with respect to $w$, we see that the singularity $(0,0,0,1)$ of $S_c$ is analytically isomorphic to 
$$(xz=y^3\subset \mathbb{A}^3)/\mu_3,$$
which is of type $\frac{1}{9}(1,2)$. 

After doing the `economic resolution' of the local model of $\mathcal{S}$ as in \cite{re87}, the fiber consists of 3 components: the birational transform $E_0$ of $S_c$ and two exceptional divisors $E_1\cong \P^2$, $E_2\cong F_2$. Furthermore, we have $E_0\cap E_1=R_1$ the $(-2)$-curve in $E_0$ and a line in $E_1$; $E_0\cap E_2=R_2$ the $(-5)$-curve on $E_0$ and of class $e+3f$ ($e$ is the class of the section with negative self-intersection and $f$ is the class of the fibers) in $E_2$ and $E_1\cap E_2=R_3$ a line in $E_1$ and the section of class $e$ in $E_2$. The configuration is illustrated in Figure 2.

\begin{center}
\setlength{\unitlength}{1mm}
\begin{picture}(100,60)
\multiput(50,30)(30,0){1}{\line(0,1){25}}
\multiput(50,30)(30,0){1}{\line(-1,-1){25}}
\multiput(50,30)(30,0){1}{\line(1,-1){25}}
{\Small
\put(43,50){$R_3$}
\put(26,12){$R_1$}
\put(71,12){$R_2$}
\put(50,15){$E_0$}
\put(65,40){$E_2\cong F_2$}
\put(25,40){$E_1\cong \P^2$}
\put(47,45){1}
\put(52,45){-2}
\put(38,20){1}
\put(41,18){-2}
\put(56,18){-5}
\put(61,20){4}
}

%\put(45, -10){\mbox{figure 3}}
\end{picture}
\end{center}
$$\mbox{Figure 2}$$

From the equation
$$wf_3(x,y,z)+f_6(x,y,z)=0$$
 of $S_c$, we know that the partial resolution $S^1\to S_c $ which only extracts $R_2$ yields a morphism $h:S^1\to \P(1,1,2)$, such that $h(R_2)=R'$ is the vanishing locus $V(f_3)$.  The morphism $h$ blows up the 9 intersection points of $R'$ with the vanishing locus $V(f_6)$. 

Each exceptional curve of $h$ gives a rational curve $F$ in $E_0$, which meets $R_2$ transversally in a single point but not $R_1$. 
Since $E_0\setminus R_1\cup R_2$ is strongly rationally connected (see \eqref{src}), the usual trick of adding teeth which are free curves in $E_0\setminus R_1\cup R_2$ to the handle  $F$ and deforming the comb (cf. \cite{kollarrc}, II.7) allows us to assume that there exists   a  very free curve  $\varphi:\P^1\to E_0$ such that $\varphi(\P^1)\cdot R_2=1$ and $\varphi(\P^1)\cdot R_1=0$. 
By considering small deformations of $\varphi$, we choose 5 such curves $F_i (1\le i\le 5)$ and construct a comb $\phi:R^c\to E_0$ with 5 teeth such that the restriction of $\phi$ to the handle is an isomorphism over $E_0$ and each tooth is mapped isomorphically to $F_i$.  Denote by $P_i=F_i\cap R_2 $, an easy computation (see \eqref{smoothing}) shows that 
$$H^1(R^c, \mathcal{N}_{\phi})\cong H^1(R_2, \mathcal{N}_{R_2}(\sum^5_{i=1} P_i ))\cong H^1(\P^1, \mathcal{O}_{\P_1})=0.$$ 
Furthermore, it follows from \eqref{smoothing} that we can smooth $\phi$ to get a morphism $\psi: \P^1\to E_0$ with $\psi(\P^1)\cdot R_2=0$ and $\psi(\P^1)\cdot R_1=1$.

To summarize we have the following statement.
\begin{prop}\label{au}
For each $R_i$ ($i=1$ or $2$), there exists a rational curve in $E_0$ meeting $R_i$ transversally in a single point but not intersecting the other component $R_{3-i}$.

\end{prop}

\begin{rmk}Besides the strong rational connectedness of $S^{sm}$, the statement of \eqref{au1} and \eqref{au} is another part of Hypothesis 14 (Key Hypothesis) in \cite{ht} when they deal with ordinary singularities. As observed in \cite{ht},
 in general there is an obstruction in the Neron-Severi group for the existence of such auxiliary curves for log Del Pezzo surfaces. For example, for the Cayley cubic surface $$S:xyz+yzw+zwx+wxy=0,$$
if we take the minimal resolution $T$, then for any $D_i$ an exceptional divisor, there does not exist any curve which meets $D_i$ transversally at one point but avoids other exceptional curves, because the sum of all $D_i$ is divisible by 2 in the Neron-Severi group.

\end{rmk}

\subsection{Moving sections}
 In this subsection, we aim to explain how we deform a given section to a new section with a given prescribed jet data.  Our approach is  similar to the one in \cite{ht06} and \cite{ht}, namely, attaching rational curves on the special fibers $T_c$   of $\mathcal{T}$ (the resolution of $\mathcal{S}$) over $c$, though we have to deal with  different configurations of divisors in the special fiber. We will do the harder case for places in $H_{II}$ and leave the argument for places in $H_{I}$ to the reader.\\

\noindent {\bf Step 1: Moving sections to $E_0$}.

In this step, we want to show that there always exists a section meeting the fiber $T_c$ in $E_0$.

Given a section $C$, by attaching sufficiently many free curves in
general fibers we can assume it is free.  

 If it meets the special fiber in $E_1$, then we can find a line $C_1$ passing through the intersection point
of $C$ with $E_1$ which we can assume meets $R_1$ and $R_3$ at general
points. We choose $C_2$ to be the ruling curve of $E_2$ containing the point $C_1\cap R_3$ and $C_3$ a curve given by
(\ref{au}) meeting $R_2$ transversally at the point  $C_2 \cap R_2$. Gluing $C$ and $C_i$ ($1\le i \le 3$) together in the obvious
way, we get a tree-like curve $f:\overline{C}\to \mathcal{T}$. Since
$\mathcal{N}_f|_{C_1}\cong \mathcal{O}_{\P^1}\oplus \mathcal{O}_{\P^1}(1)$ and
$\mathcal{N}_f|_{C_2}\cong \mathcal{O}_{\P^1}\oplus \mathcal{O}_{\P^1}$, it
follows from (\ref{smoothing}) that we can deform $\overline{C}$ to a
smooth curve, which is a section meeting $E_0$ by computing the intersection number. The configuration is illustrated in Figure 3.

\begin{center}
\setlength{\unitlength}{1.0mm}
\begin{picture}(100,60)
\put(50,30){\line(0,1){25}} \put(50,30){\line(-1,-1){25}}
\put(50,30){\line(1,0){40}} \put(13,40){\line(1,0){35}}
\put(15,35){\Small original section} {\thicklines
\put(52,52){\line(-1,-3){13}} \put(44,55){\line(2,-3){20}}}
\put(60.66,15){\line(0,1){25}}
\multiput(5,8)(2,0){23}{\line(1,0){1}}
 \put(15,10){\Small new section}

{\Small

\put(20,42){$C$} \put(55,40){$C_2$} \put(55,20){$C_3$}
\put(40,30){$C_1$}

\put(60,8){$E_0$}
\put(80,50){$E_2$}
\put(20,50){$E_1$}
\put(20,3){$R_1$}
\put(50,57){$R_3$}
\put(92,30){$R_2$}
}
$\put(50,0){Figure 3}$
\end{picture}
\end{center}

If the section $C$ meets $E_2$, then we choose a general curve $C_1$ with class $e+2f$ in $E_2$, which meets $R_2$ at three general points but not $R_3$.
Applying (\ref{au}), we can choose free curves $C_2$ and $C_3$ in $E_0$ which meet two of these three points.
Similarly, we can glue $C$ and $C_i$ ($1\le i \le 3$) together to get a tree-like curve, which can be deformed to a section meeting $E_0$. The configuration is illustrated in Figure 4.

\begin{center}
\setlength{\unitlength}{1.0mm}
\begin{picture}(100,60)
\put(50,30){\line(0,1){25}}
\put(50,30){\line(-1,-1){25}}
\put(50,30){\line(1,0){40}}
\put(30,40){\line(1,0){35.3}}

\put(17,36){\Small original section}
{\thicklines
\put(72.5,35){\line(-1,-1){20}}

\put(91.5,35){\line(-1,-1){20}}
}
\multiput(20,13)(2,0){23}{\line(1,0){1}}

\put(5,10){\Small new section}
\bezier{500}(64,50)(70,0)(80,40)
\bezier{500}(80,40)(85,45)(90,10)

{\Small

\put(45,42){$C$}
\put(80,42){$C_1$}
\put(51,20){$C_2$}
\put(80,20){$C_3$}

\put(50,0){$E_0$}
\put(80,50){$E_2$}
\put(20,50){$E_1$}
\put(23,3){$R_1$}
\put(50,57){$R_3$}
\put(92,30){$R_2$}
}

\end{picture}
\end{center}
$$\mbox{Figure 4}$$

\noindent {\bf Step 2: Moving sections out of $E_0$}.

From the last step, we know there always exists a section $C$ meeting $T_c$ in $E_0$. In this step, we show that for any component $E_i$ of $T_c$, there is a section which meets $E_i$.

 Let $C_1$ (resp. $C_2$) be a curve in $E_0$ which meets
$R_1$ (resp. $R_2$) transversally at one point and contains the point $C\cap E_0$. The existence of $C_1$ and $C_2$ is proved in (\ref{au}). Gluing
the curves together, a similar computation to the one in Step 1 using \eqref{smoothing} shows that we can smooth
it. Then the computation of the intersection number shows that the
smoothing curve is a section meeting $E_1$ (resp. $E_2$). The configuration is illustrated in Figure 5.

\begin{center}
\setlength{\unitlength}{1mm}
\begin{picture}(100,60)
\multiput(50,30)(30,0){1}{\line(0,1){25}}
\multiput(50,30)(30,0){1}{\line(-1,-1){25}}
\multiput(50,30)(30,0){1}{\line(1,0){40}}

\put(-5,6){\Small original section}
\put(-5,36){\Small new section}
\multiput(5,40)(2,0){20}{\line(1,0){1}}

\multiput(0,10)(2,0){1}{\line(1,0){48}} {\thicklines
\put(50,8){\line(-1,1){25}}}

{\Small \put(50,0){$E_0$} \put(80,50){$E_2$} \put(20,50){$E_1$} \put
(20,12){C} \put(30,29){$C_1$}}

\put(45, -10){\mbox{Figure 5}}
\end{picture}
\end{center}
%$$\mbox{Figure 5}$$

\noindent{\bf Step 3: Moving sections in the same component}

  We apply (\cite{ht}, 22) in this step.  Thus it suffices to check that for $E_i$ and a point $p\in
E_i^{sm}$, we can find a connected nodal curve $Z$ of genus zero,
distinguished smooth points $0, \infty \in Z$ in the same
irreducible component, and an unramified morphism $h$
mapping $Z$ into the special fiber with the following properties
\begin{enumerate}
\item  $h(0) = p$ and $h(\infty) = q$ a general point in $E_i^{sm}$;
\item  each irreducible component of the special fiber intersects $C$ with degree
zero;
\item $h$ takes $Z$ to the open subset of the special fiber with normal
crossings singularities of multiplicity at most two; $f^{-1}(\cup
R_i)\subset Z^{sing}$ and at points of  $R_i\cap h(Z)$ ($1 \le i\le
3$) there is one branch of $Z$ through each component of the special
fiber; and
\item $\mathcal{N}_h(-p)$ is globally generated.
\end{enumerate}

For $E_0$, this directly follows from the strong rational connectedness of  the smooth locus
$E_0\setminus(R_1\cup R_2)$ which is isomorphic to $ S^{sm}_c.$

For $E_1$, the reducible curve $Z=\cup_{1\le i\le 4}C_i$ depicted in Figure 6
will give the curve as above, where $C_1$ is a line containing $p$ and
$q$; $C_2$ is a ruling of $R_2$; $C_3$ (resp. $C_4$) is given by
(\ref{au}) which meets $R_1$ (resp. $R_2$) transversally at one point
but not $R_2$ (resp. $R_1$). The configuration is illustrated in Figure 6.

\begin{center}
\setlength{\unitlength}{1.0mm}
\begin{picture}(100,60)
\put(50,30){\line(0,1){25}} \put(50,30){\line(-1,-1){25}}
\put(50,30){\line(1,0){40}} {\thicklines
\put(52,52){\line(-1,-3){13}} \put(44,55){\line(2,-3){20}} }
\put(60.66,15){\line(0,1){25}} \bezier{500}(35,20)(42,27)(60,10)

\put(46,40){ $\bullet $ } \put(42,28){ $\bullet$ }

{\Small \put(45,42){p} \put(46,30){q}

 \put(55,40){$C_2$} \put(55,20){$C_3$} \put(39,30){$C_1$}
\put(52,10){$C_4$}

\put(50,0){$E_0$} \put(80,50){$E_2$} \put(24,50){$E_1$}
\put(20,3){$R_1$} \put(50,57){$R_3$} \put(92,30){$R_2$} }

\end{picture}
\end{center}
$$\mbox{Figure 6}$$

Similarly for $E_2$, the reducible curve $Z=\cup_{1\le i\le 4}C_i$ depicted in Figure 7 gives the
curve which we are looking for. The curve $C_1$ has the class $e+2f$ and contains $p$ and $q$;
$C_2$, $C_3$ and $C_4$ are three curves given by (\ref{au}) meeting
$R_2$ transversally at one point but not $R_1$. The configuration is illustrated in Figure 7.

\begin{center}
\setlength{\unitlength}{1.0mm}
\begin{picture}(100,60)
\put(50,30){\line(0,1){25}} \put(50,30){\line(-1,-1){25}}
\put(50,30){\line(1,0){40}}

 {\thicklines \put(72.5,35){\line(-1,-1){20}}
\put(82,35){\line(-1,-1){20}} \put(91.5,35){\line(-1,-1){20}} }

 \bezier{500}(64,50)(70,0)(80,40) \bezier{500}(80,40)(85,45)(90,10)

\put(64.2,40){$\bullet$} \put(77.6,34){$\bullet$}

 {\Small \put(62,40){p} \put(75,36){q}

\ \put(80,42){$C_1$} \put(51,20){$C_2$} \put(62,20){$C_3$}
\put(80,20){$C_4$}

\put(50,0){$E_0$} \put(80,50){$E_2$} \put(20,50){$E_1$}
\put(20,3){$R_1$} \put(50,57){$R_3$} \put(92,30){$R_2$} }
\put(45,-5){Figure 7}
\end{picture}
\end{center}

%%%%%%%%%%%%%%%%%%%%%%%%%%%%%%%%%%%%%%%%%%%%%%%%%%%%%%%%%%%%%%%%%%%%%%%%

\noindent 2-380, Massachusetts Institute of Technology,  Department of Mathematics\\
\noindent 77 Massachusetts Avenue, Cambridge, MA 02139-4307\\
\noindent {\it Email:} cyxu@math.mit.edu


\begin{thebibliography}{BDPP04}



\bibitem[CG04]{cg04}
Colliot-Th\'el\`ene, J.-L.; Gille, P.; Remarques sur l'approximation faible sur un corps de fonctions d'une variable.  {\it Arithmetic of higher-dimensional algebraic varieties (Palo Alto, CA, 2002)},  121--134, Progr. Math., {\bf 226}, {\it Birkh\"{a}user Boston, Boston, MA}, 2004.

\bibitem[Co96]{co96}
Corti, A.;
Del Pezzo surfaces over Dedekind schemes.
{\it Ann. of Math. (2)} {\bf 144} (1996), no. {\bf 3}, 641--683.

\bibitem[GHS03]{ghs}
Graber, T.; Harris, J.; Starr, J.; Families of rationally connected varieties.  {\it J. Amer. Math. Soc. } {\bf 16}  (2003),  no. {\bf 1}, 57--67 (electronic).

\bibitem[HP10]{HP10}
Hacking, P.; Prokhorov, Y.;  Smoothable Del Pezzo surfaces with quotient singularities. {\it Compos. Math.} {\bf 146} (2010), no. 1, 169-192,

\bibitem[HT06]{ht06}
Hassett, B.; Tschinkel, Y.; Weak approximation over function fields.
{\it Invent. Math.} {\bf 163} (2006), no. {\bf 1}, 171--190.

\bibitem[HT08]{ht08}
Hassett, B.; Tschinkel, Y.; Approximation at places of bad reduction for
rationally connected varieties,
{\it Pure and Applied Mathematics Quarterly} {\bf 4} (2008) no. {\bf 3}, 743-766.

\bibitem[HT09]{ht}
Hassett, B.; Tschinkel, Y.; Weak approximation for hypersurfaces of low degree.
{\it Algebraic geometry-Seattle 2005. Part 2},  937-955,  Proc. Sympos. Pure Math., {\bf 80},  Part 2, {\it  Amer. Math. Soc., Providence, RI}, 2009.

\bibitem[KMM92]{kmm}
Koll\'ar,  J.; Miyaoka, Y.; Mori, S.;
Rationally connected varieties.
{\it J. Algebraic Geom.} {\bf 1} (1992), no. {\bf 3}, 429--448.


\bibitem[Kn08]{kn}
Knecht, A.;
Weak approximation for general degree 2 Del Pezzo surfaces. arXiv: 0809.1251.


\bibitem[Ko96]{kollarrc}
Koll\'ar, J.; {\it Rational curves on algebraic varieties},
Ergeb. Math. Grenz. {\bf 3}. Folge, {\bf 32}.
Springer-Verlag, Berlin, 1996.

\bibitem[Re87]{re87}
Reid, M.;
Young person's guide to canonical singularities. {\it Algebraic geometry, Bowdoin, 1985 (Brunswick, Maine, 1985)}, 345--414, Proc. Sympos. Pure Math., {\bf 46}, Part 1, {\it Amer. Math. Soc., Providence, RI}, 1987.

%\bibitem[Te75]{te75}
%Teissier, B; Introduction to equisingularity problems.  {\it Algebraic geometry (Proc. Sympos. Pure Math., Vol. 29, Humboldt State Univ., Arcata, Calif., 1974)}, 593--632. {\it Amer. Math. Soc., Providence, R.I.,} 1975.

\bibitem[Vo03]{vo03}
Voisin, C.;
Hodge theory and complex algebraic geometry. II.
Translated from the French by Leila Schneps. Cambridge Studies in Advanced Mathematics, {\bf 77}. {\it Cambridge University Press, Cambridge}, 2003.

\bibitem[Xu08]{xu}
Xu. C.;
Strong rational connectedness of surfaces. arXiv:0810.2597. to appear in J. Reine Angew. Math..

\end{thebibliography}
\end{document}